\documentclass[11pt]{article}
  \usepackage{times}
  \usepackage{mathdots}
  \usepackage{color}
  \usepackage{amssymb,amscd,amsthm,latexsym}
  \usepackage{amsmath}
  \usepackage{graphicx}
  \graphicspath{ {graphs/} }
  \usepackage{wrapfig}
  \usepackage{float}
  \usepackage{amssymb}
  \usepackage{epstopdf}
  \usepackage{verbatim}
  \newtheorem{theorem}{Theorem}[section]
  
  \newtheorem{lemma}[theorem]{Lemma}
  \newtheorem{conjecture}[theorem]{Conjecture}

  \theoremstyle{definition}
  
  \newtheorem{example}[theorem]{Example}
  
  \newtheorem{remark}[theorem]{Remark}

  \newtheorem{cor}[theorem]{Corollary}

  \def\zy{\displaystyle{\frac{z^n-y^n}{z-y~~~}}}
  
  \def \flt{Fermat's last theorem}
  \title
  {On the prime decomposition of integers of the form $\zy$
  }
  \author{Rachid Marsli\\
  Preparatory Math Department\\
  King Fahd University of Petroleum and Minerals\\Dhahran, 31261\\
  Kingdom of Saudi Arabia\\
  rmarsliz@kfupm.edu.sa\\
  }
  \date{\today}
\begin{document}
  %
  %
  %
  \maketitle
  \begin{abstract}
  In this work, the author shows a sufficient and necessary condition for an integer of the form 
  $\zy$ to be divisible by some perfect $mth$ power $p^m$,
  where $p$ is an odd prime and $m$ is a positive integer.
  A constructive method of this type of integers is explained with details and examples.
  Links between the main result and known ideas such as Fermat's last theorem,
  Goormaghtigh conjecture and Mersenne numbers are discussed. 
  Other related ideas, examples and applications are provided.
  \end{abstract} 
  \noindent
  {\it AMS Subj. Class.:11A07 ; 11D41} 
  \begin{section}{Introduction}
  Contrary to our expectations, while we were trying to prove 
  Fermat's last theorem  by showing
  that  if $y$ and $z$ are relatively prime and $n$ and $p$ are odd prime integers,
  then $p^n$ does not divide $\zy$  , we found
  that for almost every $p$ we can construct infinitely many integers of the form
  $\zy$ each of which is divisible by $p^n$.
  Not only that, but no matter how large is the positive integer $m$, we can always construct integers of
  the form $\zy$ that are divisible by $p^m$.
  The main tool of our analysis in this work, is the concept of primitive root modulo integer.
  Given a positive integer $p$, we say that $r$ is a primitive root modulo $p$ if 
  $r$ is an integer relatively prime to $p$ and the smallest integer $a$ such that 
  $r^a \equiv 1 ~~ (mod~~p)$ is $\phi(p)$, where $\phi$ denotes the well-known Euler function.
  A positive integer possesses a primitive root if and only if
  $n= 2, 4, p^t$ or $2p^t,$  where $p$ is an odd prime and $t$ is a positive integer \cite[Theorem 8.14]{KR}.
  Another important fact about primitive roots is given by the following theorem which we state as
     a lemma for its use in the proof of the main result.
     \begin{lemma}\label{l1} \cite[Theorem 8.9]{KR}
     Let $p$ be an odd prime, then $p^k$ has a primitive root for all positive integer $k$.
     Moreover, if $r$ is a primitive root modulo $p^2$, then $r$ is a primitive root modulo $p^k$,
     for all positive integer $k$.
     \end{lemma}
  Note that there are some rare cases where a primitive root modulo $p$ is not 
  a primitive root modulo $p^2$. As an example, the prime integer $p=487$ has a primitive root $r=10$ 
  which is not a primitive root modulo $487^2$ \cite[Section 8.3]{KR}.
  More elementary ideas about primitive roots modulo integers
  can be found in number theory  textbooks such as
  \cite{DM}, \cite{KC}, \cite{KR}, \cite{NZM}, \cite{Ore} and \cite{UD}.
  Throughout the paper, the greatest common divisor of two integers $a$ and $b$ is denoted $(a,b)$.
  %
  %
  \end{section}
   %
   %
   %
    \begin{section}{Main result}
     \hskip7mm
     The following lemma is needed in the proof of the main result and contains some ideas that are
     well-known to mathematicians working on Fermat's last theorem. 
     Nevertheless, we prefer to provide a proof
     because we couldn't find a reference where all the three assertions of the lemma
     are proved together. 
     \begin{lemma}\label{l2}
     Let $y$ and $z$ be two relatively prime integers with $z\neq y$ and
     let $n$ be an odd prime integer.\\
     $~~~~~~~$ 1. If $~n~$ divides $z-y$, ~then ~ $\Big(z-y ~,~ \zy\Big) = ~n$.\\ 
     $~~~~~~~$ 2. If $~n~$ does  not divides $z-y$,
     ~then ~$n, ~(z-y)~$ and $~\displaystyle{\frac{z^n-y^n}{z-y~~~}}~$ are pairwise relatively prime.\\
     $~~~~~~~$ 3.  $~n^2~$ does  not divide $\zy$. 
    \end{lemma}
    \begin{proof}
    We have
    $$z^n = (z-y+y)^n = \sum_{i=2}^{n}\binom{n}{i}(z-y)^i\,y^{(n-i)} + n(z-y)y^{(n-1)} + y^n,$$   
    from which,
    $$
    \begin{array}{lcl}
     z^n - y^n & = & (z-y) \displaystyle{\Big[\sum_{i=2}^{n}\binom{n}{i}(z-y)^{(i-1)}\,y^{(n-i)} + ny^{(n-1)}\Big]}\\
      & = & (z-y) \displaystyle{\Big[ (z-y) \Big\{\sum_{i=2}^{n}\binom{n}{i}(z-y)^{(i-2)}\,y^{(n-i)}\Big\} + ny^{(n-1)}\Big]},
    \end{array}
    $$
    so that
    \begin{equation}\label{f32}
    \frac{z^n - y^n}{z-y} = (z-y) \Big\{\sum_{i=2}^{n}\binom{n}{i}(z-y)^{(i-2)}\,y^{(n-i)}\Big\} + ny^{(n-1)}.
    \end{equation}    
    Since $y$ and $z$ are relatively prime, the power $y^{n-1}$ and $(z-y)$ are relatively prime.
    Hence, Formula (\ref{f32}) implies that
    $$  \Big(z-y ~,~ \zy\Big) = ~n, ~~ \text{if} ~n \text{~divides~} z-y,~~~~~~~~~~~~~~~~~~~~~~~~~~~~~~~~~~~$$
    and
    $$  \Big(z-y ~,~ \zy\Big) = ~1, ~~ \text{if} ~n \text{~is relatively prime to~} z-y.~~~~~~~~~~~~~~~~~$$
    Moreover, (\ref{f32}) can be  rewritten as
    \begin{equation}\label{f33}
    \frac{z^n - y^n}{z-y} = (z-y)^{n-1} +
    \Big\{\sum_{i=1}^{n-1}\binom{n}{i}(z-y)^{(i-1)}\,y^{(n-i)}\Big\}.
    \end{equation}
    Since $n$ is a prime integer, we have
    \begin{equation}\label{fh5}
    \Big(n, \binom{n}{i}\Big) = n  \text{~~~for~~~}  i=1, 2, \dots n-1. 
    \end{equation}
     From (\ref{f33}) and (\ref{fh5}) , we get
    \begin{equation}\label{fh4}
    \zy \equiv (z-y)^{n-1}~~ (mod~n). 
    \end{equation}
    It follows from (\ref{fh4}) that if $n$ is relatively prime to $z-y$,
    then $n$ and $\zy$ are relatively prime. This is to prove the second assertion.
    The third assertion of the lemma  follows directly from the second one if $n$ does not divide $z-y$.
    Otherwise, suppose that $n$ divides $z-y$. Then from (\ref{f32})  and (\ref{fh5}),
    we can see easily that, in this case,  
    \begin{equation}\label{fh6}
     \zy~ \equiv ~ ny^{(n-1)} ~~~ (mod~~n^2).
    \end{equation}
    If $n^2$ divides $\zy$, then (\ref{fh6}) implies that $n$ divides $y$, so that also, $n$ divides $z$
    since it divides $z-y$.
    This is in contradiction  with our assumptions that $y$ and $z$ are relatively prime. 
    \end{proof}
    %
    \begin{remark}
     The first two assertions of Lemma \ref{l2} apply to the case where $n=2$,
     but the third one does not. For example, if we take 
     $z=5, y=3$ and $n=2$, then $2^2$ divides $\displaystyle{\frac{5^2-3^2}{5-3~~}}=8$.
    \end{remark}
     Next we state and prove  the main result.
    \begin{theorem}\label{t14}
     Let $y$ and $z$  be two distinct nonnegative integers
     and let $n$ be an odd prime integer.
     Let $p$ be an odd prime integer that is different than $n$ and 
     relatively prime to $y$.
     Let $r$ be a primitive root modulo $p^2$ and
     let $m$ be a positive integer. 
     Then
     $p^m$ divides $\displaystyle{\frac{z^n-y^n}{z-y~~~}}$ if and only if 
     $$n \text{~~divides~~} p-1 \text{~~~~and~~~~} 
     z ~\equiv~ y~r^{cp^{m-1}} ~~~ (mod~~p^{m}),$$
     where $c$ is any integer that satisfies:
     \begin{enumerate}
     \item $0<c<p-1$.
     \item $p-1$ divides $nc$.
     \end{enumerate}
     \end{theorem}
     \begin{proof}
     First recall that, by Lemma \ref{l1},
     $r$ is also a primitive root modulo $p^m$ for $m=1$ as well as for $m=3, 4, \dots$ 
     Suppose that  $n$ divides $p-1$ and
     \begin{equation}\label{fgg3}
     z ~\equiv~ y~r^{cp^{m-1}} ~~~ (mod~~p^{m}), 
     \end{equation}
     for some integer $c$ such that $0<c<p-1$ and $p-1$ divides $nc$.
     Formula (\ref{fgg3}) implies that 
     $z^n ~\equiv~ y^n~r^{ncp^{m-1}} ~~~ (mod~~p^{m}).$ 
     Since $p-1$ divides $nc$,
     it follows that $\phi(p^{m})$, which is equal to $(p-1)p^{m-1}$,
     divides $ncp^{m-1}$ and therefore
     \begin{equation}\label{fgg4}
     z^n ~\equiv~ y^n ~~~ (mod~~p^{m}). 
     \end{equation}
     Also, Formula (\ref{fgg3}) implies that 
     $z ~\equiv~ y~r^{cp^{m-1}} ~~~ (mod~~p)$, which is equivalent to
     $z ~\equiv~ y~r^c~r^{c(p^{m-1}-1)} ~~~ (mod~~p)$.
     Since $\phi(p)$, which is equal to $p-1$, divides $c(p^{m-1}-1)$, it follows that
     $z ~\equiv~ y~r^c ~~~ (mod~~p)$.
     By Lemma \ref{l1}, $r$ is a primitive root modulo $p$ and since $0<c<p-1$,
     we have that $r^c \not \equiv 1 ~~~(mod~~p)$.
     Hence
     \begin{equation} \label{fgg5}
     z\not\equiv y~~~ (mod~~p).
     \end{equation}
     It follows from (\ref{fgg4}) and (\ref{fgg5}) that 
     $\displaystyle{\frac{z^n-y^n}{z-y~~~}}$ is  divisible by $p^{m}$.\\
        Conversely,
        we have two different cases.
        \begin{enumerate}
        \item Case1: $z$ and $y$ are relatively prime.\\
        Suppose that $p^{m}$ divides $\displaystyle{\frac{z^n-y^n}{z-y~~~}}$.
        Then $p^{m}$ divides $z^n-y^n$ or equivalently,
        \begin{equation}\label{fgg6}
        z^n ~\equiv~ y^n~~~(mod~~p^{m}).
        \end{equation}
        Since $p$  is different than $n$ and divides $\zy$,
        Lemma \ref{l2} implies that $p$ does not divides $z-y$. Hence, 
        there exists an integer $k$ such that
        \begin{equation}\label{fhh1}
         0<k<(p-1)p^{m-1}
        \end{equation}
        and
        \begin{equation}\label{fh1}
        z ~\equiv~ y~r^{k}~~~(mod~~p^{m}). 
        \end{equation}
        This implies
        \begin{equation}\label{fgg7}
        z^n ~\equiv~ y^n r^{nk}~~~(mod~~p^{m}).
        \end{equation}
        From (\ref{fgg6}) and (\ref{fgg7}) we have 
        $y^n(1-r^{nk}) \equiv 0 ~ (mod~~p^{m})$, which leads to
        $(1-r^{nk}) \equiv 0 ~ (mod~~p^{m})$ since $y$ and $p$ are relatively prime.
        Therefore,
        \begin{equation}\label{fgg8}
         \phi(p^{m}), \text{~~which is equal to~~} (p-1)p^{m-1}, \text{~~divides~~} nk.
        \end{equation}        
        Since $p\neq n$, the above expression
        implies that
        $p^{m-1}$ divides $k$
        and because $0<k<(p-1)p^{m-1}$, there exists an integer $c$ such that $0<c<p-1$
        and
        \begin{equation}\label{fh2}
        k=c p^{m-1}. 
        \end{equation}
        From (\ref{fh2}) and (\ref{fgg8}), we have that	 $(p-1)p^{m-1}$ ~divides~ $n c p^{m-1}$.
        Thus,
        \begin{equation}\label{fh14}
         (p-1) \text{~~divides~~} nc.
        \end{equation}
        Since $0<c<p-1$ and $n$ is a prime integer,
        Formula (\ref{fh14}) implies that
        \begin{equation}\label{fh3}
        n \text{~~divides~} p-1,
        \end{equation}
        We complete the proof of this case
        by taking (\ref{fh2}) into (\ref{fh1}) to obtain
        \begin{equation}
         z ~\equiv~ y~r^{cp^{m-1}}~~~ (mod~~p^m). 
        \end{equation}
        \item Case2: $(z,y) = q >1$.\\ 
        Let $y'$ and $z'$ be such that $y = qy'$ and $z=qz'$.
        Then  $(z',y') = 1$ 
        and
        \begin{equation}
        \zy = q^{n-1}~~ \frac{z'^n-y'^n}{z'-y'}. 
        \end{equation}
        If $p^m$ divides $\zy$ with $p$ and $y$ being relatively prime, 
        then $p^m$ divides $\displaystyle{\frac{z'^n-y'^n}{z'-y'}}.$
        It follows, by Case1, that $n$ divides $p-1$ and $z' \equiv y'r^{cp^{m-1}}~~~(mod~~p^m)$,
        so that $z \equiv yr^{cp^{m-1}}~~~(mod~~p^m)$,
        where $c$ is an integer such that $0<c<p-1$ and $p-1$ divides $nc$.
        \end{enumerate}        
        \end{proof}
     \begin{remark}\label{r2}
      The integer $c$ is even and different than $\displaystyle{\frac{p-1}{2}}$.
      If $c_1$ satisfies $n~c_1 = p-1$, then the integer $c$ takes all 
      the values $c_1, c_2 = 2~c_1, c_3 = 3~c_1, \dots, c_{n-1} = (n-1)c_1$. That makes
      a total of $(n-1)$ values.
      Notice also that if $ z = y r^{c_i P^{m-1}}$ for some index $i \in \{1, 2, \dots, n-1\}$,
      then, by analogy between $z$ and $y$, we have $ y = z r^{c_j P^{m-1}}$ for some
      integer $j \in \{1, 2, \dots, n-1\}$ such that $c_i + c_j = p-1$.
      Moreover, $c_i \neq c_j$ for if they were equal,
      then we would have  $c = \displaystyle{\frac{p-1}{2}}$, which is impossible as is already mentioned.      
     \end{remark}
     \begin{remark}\label{r3}
       Note that, in the statement of Theorem \ref{t14},  the condition $p\neq n$
       needs to be stated for the case $m=1$ only.
      If $m \geq 2$ and $p^m$ divides $\zy$,
      then the third assertion of Lemma \ref{l2} ensures that $p \neq n$. 
     \end{remark}
     \begin{remark}\label{r1}
     Observe that $n$ divides $\displaystyle{\frac{p-1}{2}}$ in Theorem \ref{t14}. Therefore, 
     if $p < 2n+1$, then $p^m$ does not divide $\zy$.
     \end{remark}
     \begin{remark}
      If an odd prime $q$  divides $\zy$ but $n$ does not divide $q-1$,
      then by Theorem \ref{t14} and Lemma \ref{l2}, $t$ is equal to $n$ and divides $z-y$.
     \end{remark}
     \begin{example}
      Goormaghtigh conjecture states that the Diophantine equation
      $$ \frac{x^{n_1}-1}{x-1}=\frac{y^{n_2}-1}{y-1}, ~~~ x>y>1 \text{~~and~~} n,m>2 ,$$
      is satisfied for only two trivial cases:
      $$\frac{5^{3}-1}{5-1} = \frac{2^{5}-1}{2-1} = 31$$
      and
      $$\frac{90^{3}-1}{90-1} = \frac{2^{13}-1}{2-1} = 8191.$$
      The condition imposed by Theorem \ref{t14} that $n$ divides $p-1$ is satisfied in both cases.
      In the first case, we have $n_1=3$ divides $p-1 = 30 = (2)(3)(5)$.
      In the second case, $p=8191$ is a prime number and
      each of $n_1=3$ and $n_2=13$ divides $p-1=8190=(3^2)(7)(13)$.      
     \end{example}
     \begin{example}
      A Mersenne number is an integer of the form  $2^n-1$.
      Therefore, it is of the form $\zy$.
      It is well-known that if a prime $p$ divides $2^n-1$, where $n$ is an odd prime, 
      then $n$ divides $p-1$. This fact is in accordance with Theorem \ref{t14}.
      It means that for every odd prime integer $n$,
      there is another prime integer $p$ strictly larger than $n$. As it is known,
      This idea implies the infinitude of prime integers. 
      \end{example}
      Two particular cases of Theorem \ref{t14} are $m=1$ and $m=n$.  we state the second one  as a corollary
      because of its connection with \flt.
      \begin{cor} \label{fb1}
      Let $y$ and $z$  be two distinct nonnegative integers.
      Let $p$ be an odd prime integer relatively prime to $y$ and let $r$ be a primitive root modulo $p^2$. 
      and let $n$ be an odd prime integer. Then
      $p^n$ divides $\displaystyle{\frac{z^n-y^n}{z-y~~~}}$ if and only if 
      $$n \text{~~divides~~} p-1 \text{~~~~and~~~~} 
      z ~\equiv~ y~r^{cp^{n-1}} ~~~ (mod~~p^{n}),$$
      where $c$ is any integer that satisfies:
      \begin{enumerate}
      \item $0<c<p-1$.
      \item $p-1$ divides $nc$.
      \end{enumerate}
      \end{cor}
      %
      \begin{cor}\label{c2}
      Let $y$ and $z$  be two distinct nonnegative integers 
      and let $n$ be an odd prime integer.
      Let $p$ be an odd prime integer different than $n$, relatively prime to $y$
      and having the form $p = 2^k +1$ for some positive integer $k$.
      Then $p$ does not divide $\zy$.
      \end{cor}
      \begin{proof}
      Follows, immediately, from Theorem \ref{t14}
      since there is no odd prime integer $n$ that divides $p-1 = 2^k$.
      \end{proof}
      %
      As a completion of Theorem \ref{t14}, we show that integers of the form $\zy$ are not divisible by $2$,
      given that $z$ and $y$  are not both even and $n$ is an odd prime integer.
      \begin{theorem}\label{t12}
      Let $y$ and $z$ be two  distinct nonnegative integers not both even
      and let $n$ be an odd prime integer.
      Then $2$ does not divide $\displaystyle{\frac{z^n - y^n}{z-y~~}}$.
      \end{theorem}
      \begin{proof}
      It suffices to show that $\zy$ is an odd integer.
      If one of $y$ and $z$ is odd and the other  is even, then both $(z^n-y^n)$ and $(z-y)$ are odd integers.
      Hence, their quotient $\zy$ is also odd.
      If each of $y$ and $z$ is odd, then $(z-y)$ is even. Hence,
      $\zy$ has to be an odd integer since, by Lemma \ref{l2}, ~$\Big(\zy, z-y\Big) = 1$ or $n$. 
      \end{proof}
    \end{section}
    \begin{section}{Some applications of Theorem \ref{t14}} 
    \begin{subsection}{Construction of  integers having the form $\zy$  and divisible by $p^m$}
    Theorem \ref{t14}, beside being a characteristic theorem, it is also a constructive theorem.
    In other words, if $y, p, n, m, $ are as in theorem \ref{t14}, $\displaystyle{c_1=\frac{p-1}{n}}$
    and $r$, is a primitive root modulo $p^2$,  
    then we can construct the set $\xi(y,p,n,m,c_1)$ of all  integers of the form $\zy$ that are divisible by $p^m$,
    \begin{equation}\label{fh13}
    \xi(y,p,n,m,c_1) = \Big\{ \zy  ~~\Big|~~~ z \equiv r^{c_1~p^{m-1}}~~(mod~~p^m)\Big\}. 
    \end{equation}
     As we have explained in Remark \ref{r2}, the integer $c_1$ can be replaced by $c_i = i\, c_1$
     for $i =1, 2, \dots, n-1$, so that we can construct sets of the form:
     \begin{equation}\label{fh13}
    \xi(y,p,n,m,c_i) = \Big\{ \zy  ~~\Big|~~ z \equiv r^{i\,c~p^{m-1}}~~(mod~~p^m)\Big\}, ~~~ i=1, 2, \dots, n-1. 
     \end{equation}
     The union, over $i$, of the above sets is
     \begin{equation}
     \xi(y,p,n,m) = \bigcup_{i=1}^{n-1} ~ \xi(y,p,n,m,c_i).  
     \end{equation}
     Let $\xi(p,n,m)$ be the set of all integers of the form $\zy$ that are divisible by $p^m$,
     $y$ relatively prime to $p$, Then $\xi(p,n,m)$
     is obtained by taking the union of the sets of the form $\xi(y,p,n,m)$
     over  all possible values of $y$.
     \begin{equation}
     \xi(p,n,m)  = \bigcup_{\substack{y \in \mathbb{N}\\p ~\nmid~ y}}~\xi(y,p,n,m)
     = \bigcup_{\substack{y \in \mathbb{N}\\ p ~\nmid~ y}} \bigcup_{i=1}^{n-1} ~ \xi(y,p,n,m,c_i).  
     \end{equation}
     \begin{remark}
      Unless $p=3$, there are many primitive roots that are incongruent modulo $p^2$.
      However, we do not consider $r$ to be a parameter in the construction of  $\xi(p,n,m)$ since
      this set remains invariant if we replace  $r$ by another primitive root modulo $p^2$. 
      This can be easily verified.
      \end{remark}
     Suppose that $\zy \in \xi(y,p,n,m,c=c_2)$. Then $$z \equiv y~r^{c_2\,p^{m-1}} ~~~ (mod~~p^m).$$
     Since $c_2 = 2~c_1$,
     the above congruence equation can be rewritten as
     $$z \equiv \big(y~r^{c_1\,p^{m-1}}\big)~r^{c_1\,p^{m-1}} ~~~ (mod~~p^m).$$
     Letting $y' = y~r^{c_1\,p^{m-1}}$, we obtain
     $$z \equiv y'~r^{c_1\,p^{m-1}} ~~~ (mod~~p^m),$$
     so that $\displaystyle{\frac{z^n - y'^n}{z-y'}} \in \xi(y',p,n,m,c_1).$
     The above reasoning shows that
     \begin{equation}
     \bigcup_{\substack{y \in \mathbb{N}\\p ~\nmid~ y}} \bigcup_{i=1}^{n-1} ~ \xi(y,p,n,m,c_i)
     =\bigcup_{\substack{y\in \mathbb{N}\\p ~\nmid~ y}}  ~ \xi(y,p,n,m,c_1).
     \end{equation}  
     Therefore, we have the following corollary.
     \begin{cor}\label{c9}
      Let $p$ be an odd prime integer for which there exists  an other odd prime integer $n$  
      such that $p-1 = n\,c$ for some positive integer $c$. Let $r$ be a primitive root modulo $p^2$.
      Then
      \begin{equation}\label{fb2}
      \xi(p,n,m) =  \bigcup_{\substack{y\in \mathbb{N}\\p \nmid y}}
      \Big\{ \zy  ~~\Big|~~ z \equiv r^{c\,p^{m-1}}~~~(mod~~p^m)\Big\}
      \end{equation}
      is the set of all integers of the form $\zy$ that are divisible by $p^m$,
      where $p$ does not divide $y$ and $m$ is a positive integer.
     \end{cor}
     \begin{remark}
      Notice that no matter how the integer $m$ is large, we can construct infinitely many integers
     of the form $\zy$ divisible by $p^m$.
     Notice also that 
     $$ \xi(p,n,m) \subseteq \xi(p,n,m'),   \text{~~~for~~}  1\leq m' < m .$$
     \end{remark}
     \begin{example}\label{e1}
     Let's construct an integer of the form $\displaystyle{\frac{z^3-y^3}{z-y}}$ that is divisible by $7^3$.
     Take $p=7, r=3, n=3, c=2$ and  $y=1$.
     We have that  $n=3$ divides $p-1=6$ and $nc=6=p-1$. 
     Construct the integer $z=r^{c\,p^{n-1}} = 3^{98}$.
     Then, by Theorem \ref{t14},
     $$ 7^3 = 343 \text{~~divides~~} \frac{(3^{98})^3-1}{3^{98}-1}.$$
     Of course, this is a huge number. But Theorem \ref{t14} ensures that we can use 
     positive numbers that are less than and equivalent to $z$ modulo $p^n$.
     By the use of a calculator, we find easily that $3^{98} \equiv 324~~(mod~7^3)$.
     Indeed, $$\frac{324^3-1}{324-1} = 105301 = (307)(7^3).$$
    \end{example}
    Now, let's ask a question:\\
    Is it true that, for an odd prime integer $n$,
    there are infinitely many odd prime integers $p$ such that $n$ divides $p-1$?\\
    Consider the set
    $$\xi(y=1,p,n,m=1) = \Big\{ \frac{z^n-1}{z-1} ~~|~~ z \equiv r^c ~~~(mod~~p)\Big\}.$$
    and let $E$ be the set of all odd prime integers $p$
    such that $p$ divides some element from $\xi(y=1,p,n,m=1)$.
    Since, by Theorem \ref{t14}, $n$ divides $p-1$ for every element $p \in E$,
    an affirmative answer of the above question can be obtained if we prove that
    there are infinitely many element in $E$. This seems to be true because
    every two element of $\xi(y=1,p,n,m=1)$ have, more likely, different prime decomposition.    
    \end{subsection}
   %
   \begin{subsection}{Proving a general fact about the congruence modulo $p^m$}
    Beside its constructive aspect, Theorem \ref{t14} has other applications such as the following.
   \begin{cor}\label{c4}
    Let $p$ be an odd prime integer
    for which there exist another prime integer $n$ such that $n$ divides $p-1$.
    Let $r$ be a primitive root modulo $p^2$.    
    Let $c$ be an integer such that $0<c<p-1$ and $p-1$ divides $nc$.
    Then, for every positive integer $m$, we have
    \begin{equation}\label{e4}
     \sum_{k=0}^{n-1} r^{kcp^{m-1}} \equiv 0 ~~~ (mod~~p^m).
    \end{equation}
    In particular, for $m=1$, we have
    \begin{equation}\label{e5}
     \sum_{k=0}^{n-1} r^{kc} \equiv 0 ~~~ (mod~~p).
    \end{equation}
    \end{cor}
   \begin{proof}
    We choose an integer $y$ relatively prime to $p$, and we
    construct the integer  
    \begin{equation}\label{fh7}
     z = y r^{cp^{m-1}}.
    \end{equation}
    By Theorem \ref{t14}, we have $\zy \equiv 0 ~~~ (mod ~~p^m)$, which is equivalent to
    \begin{equation}\label{fh8}
     \sum_{k=0}^{n-1}~z^k~y^{n-k-1} \equiv  0 ~~~ (mod ~~p^m).
    \end{equation}
    Taking (\ref{fh7}) into (\ref{fh8}), we obtain
    \begin{equation}\label{fh9}
     y^{n-1}~\sum_{k=0}^{n-1}~r^{kcp^{m-1}} \equiv  0 ~~~ (mod ~~p^m).
    \end{equation}
    Since $y$ and $p$ are relatively prime, it follows from (\ref{fh9}) that
    \begin{equation}
     \sum_{k=0}^{n-1}r^{kcp^{m-1}} \equiv  0 ~~~ (mod ~~p^m).
    \end{equation}
   \end{proof} 
   \begin{remark}
    It is well-known that if $r$ is primitive root mod $p$, then 
    \begin{equation}\label{e3}
     \sum_{k=0}^{p-1}~r^k ~ \equiv  0 ~~~ (mod ~~p).
    \end{equation}
    To see this, recall that $r^1, r^2, \dots, ..., r^{n-1}$ form a complete residue set modulo $p$.
    A question that arises is:
    do we have similar formula for an integer $t$ that is not a primitive root modulo $p$?
    The above  corollary gives a partial answer to this question
    by the mean of Formula (\ref{e5}) which can be considered as an extension of Formula (\ref{e3}).
    In fact, if  $ t = r^{c}$, then $t$ is not a primitive root modulo $p$
    since $0<c<p-1$ and $\big(c,p-1\big)\neq 1$. Then Formula (\ref{e5}) becomes
    \begin{equation}\label{e6}
     \sum_{k=0}^{n-1} t^k \equiv 0 ~~~ (mod~~p).
    \end{equation}
    Note that $n < \displaystyle{\frac{p}{2}}$. That is, the number of summands in (\ref{e6}) is less than 
    half of that in (\ref{e3}).
   \end{remark}
   \begin{example}\label{e2}
    As in Example \ref{e1}, we take $p=7, r= 3, n=3$ and $c=2$.
    If we let $m=n=3$, then we have
    $$
    \begin{array}{ccl}
     \sum_{k=0}^{n-1}r^{kcp^{n-1}} & = & \sum_{k=0}^{2}~3^{98k}\\
     & = & 1+3^{98}+3^{196}\\
     & \equiv & 1 + 324 + 324^2~~~(mod~~7^3)\\
     & \equiv & 1 + 324 + (-19)^2~~~(mod~~343)\\
     & \equiv & 1 + 324 + 361 ~~~(mod~~343)\\
     & \equiv & 0~~~(mod~~7^3)\\
    \end{array}
    $$
    By the same reasoning, if $m=1$, then
    $$\sum_{k=0}^{2} 3^{kc} = 3^0+3^2+3^4 = 91 \equiv 0 ~~~ mod~~7.$$
    An other primitive  root of $7$ is the integer $5$. For $m=1$, we have
    $$\sum_{k=0}^{2} 5^{kc} = 5^0+5^2+5^4 = 651 \equiv 0 ~~~ mod~~7.$$
    \end{example}    
    \end{subsection}
    \begin{subsection}{Case where  the $mth$ power of a composite integer divides $\zy$}
    Let $p_1, p_2, \dots, p_k$ be $k$ distinct prime integers each of which is different
    than $n$ and relatively prime to $y$.
    Suppose that the product $\displaystyle{\prod_{i=1}^{k} p_i^{m_i}}$ divides $\zy$,
    where $m_1, m_2, \dots, m_k$ are $k$ positive integers.
    According to Theorem \ref{t14}, this hold if and only if $n$ divides $p_i -1~$	
    for $~i=1, 2, \dots, k~$ and  
    \begin{align*}
    &z \equiv y ~r_1^{c_1 p_1^{m_1-1}} ~~~ (mod~~p_1^{m_1})&\\
    &z \equiv y ~r_2^{c_2 p_2^{m_2-1}} ~~~ (mod~~p_2^{m_2})&\\
    &\dots&\\
    &z \equiv y ~r_k^{c_k p_k^{m_k-1}} ~~~ (mod~~p_k^{m_k}),&\\
    \end{align*}
    where, for $i=1, 2, \dots, k$, $r_i$ is a primitive root modulo $p_i$,
    the integer $c_i$ satisfies $0<c_i< p_i-1$ and $p_i-1$ divides $nc_i$.
    By the Chinese remainder theorem, the above system of congruence equations holds if and only if
    \begin{align}
     & z & \equiv & & \sum_{i=1}^k~ \Big(y~r_i^{c_i p_i^{m_i-1}}\Big)~\Big(M_i~q_i\Big)
     ~~~ (mod~~p_1^{m_1}~p_2^{m_2}\dots p_k^{m_k})& &     \notag \\
     &  & \equiv & & y ~ \sum_{i=1}^k~ M_i~q_i ~ r_i^{c_i p_i^{m_i-1}} ~~~ 
     (mod~~p_1^{m_1}~p_2^{m_2}\dots p_k^{m_k}),& &
    \end{align}
    where $M_i = \displaystyle{\frac{\prod_{j=1}^k p_j^{m_j}}{p_i^{m_i}}}$ and
    $q_i$ is any integer that satisfies $M_i q_i  \equiv 1 ~~~ (mod~~p_i^{m_i}).$
    The following corollary summarize the above result.
    \begin{cor}\label{c8}
     Let $y$ and $z$  be two relatively prime integers and let $n$ be an odd prime integer.
     Let $p_1, p_2, \dots, p_k$ be $k$ distinct odd prime integers, 
     each of which  is different than $n$.
     and let $r_1, r_2, \dots, r_k$ be, respectively, primitive root modulo $p_1^2, p_2^2, \dots, p_k^2$. 
     Let $m_1, m_2, \dots, m_k$ be $k$ positive integers.
     Then the product $\displaystyle{\prod_{i=1}^k~ p_i^{m_i}}$ divides $\zy$ if and only if
     \begin{equation}
      n~~~\text{divides}~~~ p_i-1, ~~~\text{for}~~~ i=1, 2, \dots, k,
     \end{equation}
     and
     \begin{equation}
      z  \equiv y ~ \sum_{i=1}^k~ M_i~q_i ~ r_i^{c_i p_i^{m_i-1}} ~~~ 
     (mod~~p_1^{m_1}~p_2^{m_2}\dots p_k^{m_k}),
     \end{equation}
     where
     \begin{enumerate}
     \item $M_i = \displaystyle{\frac{\prod_{j=1}^k p_j^{m_j}}{p_i^{m_i}}},$\\
     \item $q_i$ is any integer that satisfies $M_i q_i  \equiv 1 ~~~ (mod~~p_i^{m_i}),$\\
     \item $c_i$ is an integer such that $0 < c_i < p_i -1$ and $p_i-1$ divides $nc_i$.     
     \end{enumerate}
     \end{cor}
     \end{subsection}   
   \end{section}
   \begin{section}{Connection with Fermat's last theorem}
   Fermat's last theorem \cite{TW} states:
   \begin{theorem}\label{flt}
   For every positive integer $n$ with $n \geq 3$, no positive integers $x, y$ and $z$ satisfy
   $$z^n = x^n + y^n.$$
   \end{theorem}
          %
    This theorem, which has been proved around 1995 \cite{TW},  implies the following fact.
    \begin{cor}\label{c7}
     Let $z$ and $y$ be two relatively prime integers, and let $n$ be an odd prime integer.
     Then $z-y$ is a perfect $nth$ power if and only if $\zy$ is not a perfect $nth$ power.
     In particular, if $z-y=1$, then $\zy$ is not an $nth$ perfect power.
    \end{cor}
    We have proved in this work, that $\zy$ can be  multiple of some perfect $nth$ power $p^n$.
    But we don't know if $\zy$, itself, can be a perfect $nth$ power having the form
    $(p_1p_2\dots)^n$ for not necessary distinct prime integers $p_1, p_2, \dots$
    By going back to Formula (\ref{fb2}) and looking at how large is the set $\xi(p,n,m)$
    and the degree of freedom that we have to construct such set
    by acting on different parameters $p$ and $n$,
    one may believe that there is chance for
    some elements of $\xi(p,n,m)$ to be perfect $nth$ powers.
    For instance, consider  the number
    $a= \displaystyle{\frac{16^3 - 5^3}{16-5}}$ which is equal to  $19^2$.
    Of course, the integer $a$ is not a perfect $nth$ power since $n=3$. But it is well
    a perfect power. Moreover, it is a perfect power of a prime integer.
    Since such number $a$ exists and it is remarkably small, we believe that nothing impede 
    the existence of a perfect $nth$ power of the form $\zy$.
    However, it may turn out that the smallest of these numbers is tremendously large and
    therefore difficult to reach with a computer.
    Perhaps, a constructive proof, is the best way to find such integers if they exist.\\
    For mathematicians seeking a proof of \flt\-  by the  mean of classical methods,
    we have a little result that may be of some use
    and which is consequence of Theorem \ref{t14}.
    \begin{theorem}
     Suppose that there are pairwise relatively prime positive integers $x, y, z$
     such that $z^n = x^n+y^n$, where $n$ is an odd prime integer.
     If $p$ is an odd prime integer such that $p\neq n$
     and $p$ divides $\displaystyle{\frac{x^ny^nz^n}{(z-x)(z-y)(x+y)}}$, then
     $n$ divides $p-1$.      
    \end{theorem}
    \begin{proof}
     $p$ divides one of $\displaystyle{\frac{x^n}{z-x}}$, $\displaystyle{\frac{y^n}{z-y}}$
     and $\displaystyle{\frac{z^n}{x+y}}$.
     Then, by Theorem \ref{t14}, $n$ divides $p-1$.
    \end{proof}
   \end{section}
   \begin{section}{Conclusion}
   We believe that a lot can be done about the integers of the form $\zy$.
   A better understanding of this type of integers may lead 
   to a more accessible proof of \flt\ as well as
   to the solutions of other Diophantine equations.
   For instance, an observations made on some few  integers 
   of the form $\zy$ gives us the impression that
   there may be always some prime integer $p$ divisor of $\zy$ that is
   greater than each of $|z|$ and $|y|$.
   We strongly believe that this observation holds for all integers
   of the form $\zy$. Therefore we state it as a conjecture.
   \begin{conjecture}
    Let $z$ and $y$ be two relatively prime positive integers with $z>y$
    and let $n$ be an odd prime integer.
    There is a prime integer $p$ divisor of $\zy$ such that $p>z$.
   \end{conjecture}
    If it happens that this conjecture is true, then \flt\, will be an immediate consequence of it.   
    \begin{center}
	      \begin{tabular}{ | c | c | c |c|c|c|c|c|c|c|c|}
		\hline
		    $y$ &   $z$    &$\displaystyle{\frac{z^3-y^3}{z-y~}}$    & prime decomp &  $p, ~ p>z$ & &$y$ &   $z$ &$\displaystyle{\frac{z^5-y^5}{z-y~}}$      & prime decomp   &  $p, ~ p>z$  \\ \hline  	
	            $5$ &   $6$    &     $91$     &     $7*13$     &    $13$     & &$7$ &   $8$     &     $15961$     &     $11*1451$   &    $1451 $    \\ \hline
	            $5$ &   $7$    &     $109$    &     prime      &    $109$    & &$7$ &   $9$     &     $21121$     &     prime       &    $21121$    \\ \hline
	            $5$ &   $8$    &     $129$    &     $3*43$     &    $43$     & &$7$ &   $10$    &     $27731$     &     $11*2521$   &    $2521$     \\ \hline      
	            $5$ &   $9$    &     $151$    &     prime      &    $151$    & &$7$ &   $11$    &     $36061$     &     prime       &    $36061$    \\ \hline
	            $5$ &   $11$   &     $201$    &     $3*67$     &    $67$     & &$7$ &   $12$    &     $46405$     &     $5*9281$    &    $9281$     \\ \hline
	            $5$ &   $12$   &     $229$    &     prime      &    $229$    & &$7$ &   $13$    &     $59081$     &     $11*41*131$ &    $131, 41$  \\ \hline
	            $5$ &   $13$   &     $259$    &     $7*37$     &    $37$     & &$7$ &   $15$    &     $92821$     &     prime       &    $92821$    \\ \hline
	            $5$ &   $14$   &     $291$    &     $3*97$     &    $97$     & &$7$ &   $16$    &     $114641$    &     prime       &    $114641$   \\ \hline
	            $5$ &   $16$   &     $361$    &     $19*19$    &    $19$     & &$7$ &   $17$    &     $140305$    &     $5*11*2551$ &    $2551$     \\ \hline
	            $5$ &   $17$   &     $399$    &     $3*7*19$   &    $19$     & &$7$ &   $18$    &     $170251$    &     $61*2791$   &    $2791$     \\ \hline
	            $5$ &   $18$   &     $439$    &     prime      &    $439$    & &$7$ &   $19$    &     $204941$    &     $11*31*601$ &    $601, 31$  \\ \hline
	            $5$ &   $19$   &     $481$    &     $13*37$    &    $37$     & &$7$ &   $20$    &     $244861$    &     prime       &    $244861$   \\ \hline
	            $5$ &   $21$   &     $571$    &     prime      &    $571$    & &$7$ &   $22$    &     $342455$    &     $5*68491$   &    $68491$    \\ \hline
	            $5$ &   $22$   &     $619$    &     prime      &    $619$    & &$7$ &   $23$    &     $401221$    &     $71*5651$   &    $71, 5651$ \\ \hline
	            $5$ &   $23$   &     $669$    &     $3*223$    &    $223$    & &$7$ &   $24$    &     $467401$    &     $11 x 42491$&    $42491$    \\ \hline
	            $5$ &   $24$   &     $721$    &     $7*103$    &    $103$    & &$7$ &   $25$    &     $541601$    &     $31 x 17471$&    $31, 17471$\\ \hline
	            $5$ &   $26$   &     $831$    &     $3*277$    &    $277$    & &$7$ &   $26$    &     $624451$    &     prime       &    $624451$   \\ \hline
	            $5$ &   $27$   &     $889$    &     $7*127$    &    $127$    & &$7$ &   $27$    &     $716605$    &     $5*251*571$ &    $251, 571$ \\ \hline
	            $5$ &   $28$   &     $949$    &     $13*73$    &    $73$     & &$7$ &   $29$    &     $931561$    &     $41*22721$  &    $41, 22721$\\ \hline
	            $5$ &   $29$   &     $1011$   &     $3*337$    &    $337$    & &$7$ &   $30$    &     $1055791$   &     $11*41*2341$&    $41, 2341$ \\ \hline
	            $5$ &   $31$   &     $1141$   &     $7*163$    &    $163$    & &$7$ &   $31$    &     $1192181$   &     prime       &    $1192181$  \\ \hline
	      \end{tabular}
	      \vskip3mm
	      Table1: The prime decomposition of some small numbers of the forms $\zy$.\\
	        Each one of them has a prime divisor that is larger than $z$.
	    \end{center}
   \end{section}
   {\bf \Large Acknowledgements}
    \vskip3mm
    The author would like to thank Abdullah Laaradji from King Fahd University of Minerals and Petroleum 
    for his very useful comments and suggestions.\\
   %
        
   %
   \end{document}